\documentclass[10pt,twoside]{article}
\usepackage{}
\usepackage{amsfonts}
\usepackage{fancyhdr}
\usepackage{titlesec}
\usepackage{CJK}
\usepackage{cite}
\usepackage{ifthen}
\usepackage{pifont}
\usepackage{stmaryrd}
\usepackage{setspace}
\usepackage{indentfirst}
\usepackage{latexsym,bm}
\usepackage{amsmath,amssymb,amscd,bbm,amsthm,mathrsfs,dsfont}
\input amssym.def

\titleformat{\section}{\centering\large\bfseries}{\S\arabic{section}}{1em}{}
\newboolean{first}
\setboolean{first}{true}

\newtheorem{theorem}{Theorem}[section]

\newtheorem{proposition}{Proposition}[section]

\numberwithin{equation}{section}
\begin{document}
\begin{CJK*}{GBK}{song}
\setlength\abovedisplayskip{2pt}
\setlength\abovedisplayshortskip{0pt}
\setlength\belowdisplayskip{2pt}
\setlength\belowdisplayshortskip{0pt}
\title{\bf \Large Well-posedness of a porous medium flow with fractional pressure in Sobolev spaces\author{Xuhuan Zhou$^1$\ \ \  Weiliang Xiao$^{2*}$ }\date{}} \maketitle
 \footnote{Mathematics Subject Classification(2000): 35K55, 35K65, 76S05.}
 \footnote{Keywords: fractional porous medium equation, degenerate diffusion transport equation, Sobolev spaces.}
\begin{center}
\begin{minipage}{135mm}
{\bf \small Abstract}.\hskip 2mm {\small
 The nonnegative solution for a linear degenerate diffusion transport eqution is proved. As a result, we show the existence and uniqueness of the solution for
  the fractional porous medium equation in Sobolev spaces $H^\alpha$ with nonnegative initial data, $\alpha>\frac d2+1$}. Besides, we correct a mistake in our previous paper \cite{zhou01}.
\end{minipage}
\end{center}
\thispagestyle{fancyplain} \fancyhead{}

\section{Introduction}
 In this paper we consider the following porous medium type equation
\begin{equation}\label{equation01}
\partial_t u=\nabla\cdot(u\nabla p),\ \ p=(-\Delta)^{-s}u,\ \  0<s<1.
\end{equation}
Where $x\in \mathbb{R}^n$, $n\geq 2$, and $t>0$. The initial data $u(x,0)\geq 0$.

This model is based on Darcy's law and the pressure is given by an inverse fractional Laplacian operator. It was first introduced by Caffarelli and V\'azquez \cite{caffarelli02}, in which they proved the existence of a weak solution when $u_0$ is a bounded function with exponential decay at infinity. For $\alpha=\frac{n}{n+2-2s}$, Caffarellin, Soria and V\'azquez \cite{caffarelli01} proved that the bounded nonnegative solutions are $C^\alpha$ continuous in a strip of space-time for $s\neq1/2$. And same conclusion for the index $s=1/2$ was proved by Caffarelli and V\'azquez in \cite{caffarelli03}.
\cite{carrillo01,caffarelli04,vazquez01} give a detail description of the large-time asymptotic behaviour of the solutions of (\ref{equation01}).
\cite{biler01, stan01} considered some degenerate cases and show the existence and properties of self-similar solutions.
Allen, Caffarelli and Vasseur \cite{allen01} studied the equation with another fractional time derivative, and proved the H\"older continuity for its weak solutions.

In this paper, we study the existence and uniqueness of solutions of (\ref{equation01}) in Sobolev spaces. The method we used here is novel: Unlike the usual
way to consider the weak solution in $L^{\infty}$ or construct approximate solutions of linear transport systems, we solve equation (\ref{equation01}) by constructing linear degenerate diffusion transport systems. The well-posedness and properties of the constructed linear degenerate diffusion transport are interesting problem themselves. By this way we get that for $s\in [\frac12,1)$, $\alpha>\frac d2+1$, $u_0\in H^\alpha(\mathbb{R}^n)$ nonnegative, then  for some $T_0>0$, the unique solution of (1.1) in $\mathbb{R}^n\times[0,T_0]$ exists. Besides, using the methods and results in this paper, we correct a mistake in our previous paper \cite{zhou01}.

\section{Preliminaries}
Define $\rho(x)\in C_c^\infty(\mathbb{R}^n)$ by
\begin{equation*}
\rho(x)=
\begin{cases}
c_0\exp(-\frac 1{1-|x|^2}), |x|<1,\\
0, |x|\geq 1,
\end{cases}
\end{equation*}
where $c_0$ is selected such that $\int \rho(x)dx=1$.
The operator
$J_\epsilon$ is defined by
\begin{equation*}
J_\epsilon u=\rho_\epsilon*u=\epsilon^{-n}\rho(\frac\cdot\epsilon)*u,
\end{equation*}
and it has the following properties:
\begin{proposition}
(1) $\Lambda^s J_\epsilon u= J_\epsilon \Lambda^s u$, $s\in \mathbb{R}$.\\
(2) For all $u\in L^p(\mathbb{R}^n)$, $v\in H^\alpha(\mathbb{R}^n)$, with $\frac 1p+\frac 1q=1$,
$\int (J_\epsilon f)g=\int f(J_\epsilon g)$.\\
(3) For all $u\in H^\alpha(\mathbb{R}^n)$,
\begin{equation*}
\lim_{\epsilon\rightarrow 0}\|J_\epsilon u-u\|_{H^\alpha}=0,\ \ \lim_{\epsilon\rightarrow 0}\|J_\epsilon u-u\|_{H^{\alpha-1}}\leq C\|u\|_{H^\alpha}.
\end{equation*}
(4) For all $u\in H^\alpha(\mathbb{R}^n)$, $s\in \mathbb{R}$, $k\in \mathbb{Z}\cup\{0\}$, then
\begin{equation*}
\|J_\epsilon u\|_{H^{\alpha+k}}\leq \frac {C_{\alpha k}}{\epsilon^k}\|u\|_{H^\alpha},\ \
\|J_\epsilon D^ku\|_{L^\infty}\leq \frac {C_{k}}{\epsilon^{\frac n2+k}}\|u\|_{H^\alpha}
\end{equation*}
\end{proposition}
Following propositions can be found in \cite{cordoba01,ju01}.
\begin{proposition}
Suppose that $s>0$ and $1<p<\infty$. If $f,g\in \mathcal {S}$, the Schwartz class, then we have
\begin{equation*}
\|\Lambda^s(fg)-f\Lambda^s g\|_{L^p}\leq c\|\nabla f\|_{L^{p_1}}\|g\|_{\dot{H}^{s-1,p_2}}+c\|g\|_{L^{p_4}}\|f\|_{\dot{H}^{s,p_3}}
\end{equation*}
and
\begin{equation*}
\|\Lambda^s(fg)\|_{L^p}\leq c\|f\|_{L^{p_1}}\|g\|_{\dot{H}^{s,p_2}}+c\|g\|_{L^{p_4}}\|f\|_{\dot{H}^{s,p_3}}
\end{equation*}
with $p_2,p_3\in(1,+\infty)$ such that $\frac 1p=\frac 1{p_1}+\frac 1{p_2}=\frac 1{p_3}+\frac 1{p_4}$.
\end{proposition}
\begin{proposition}
Let $0\leq s\leq 2$, $f\in \mathcal {S}(\mathbb{R}^n)$, we have the pointwise inequality,
\begin{equation*}
2f(x)\Lambda^s f(x)\geq \Lambda^s f^2(x).
\end{equation*}
\end{proposition}
\begin{proposition}
Let $\alpha_1$ and $\alpha_2$ be two real numbers such that $\alpha_1<\frac n2$, $\alpha_2<\frac n2$ and $\alpha_1+\alpha_2>0$. Then there exists a const $C=C_{\alpha_1,\alpha_2}\geq 0$ such that for all $f\in \dot{H}^{\alpha_1}$ and $g\in \dot{H}^{\alpha_2}$,
\begin{equation*}
\|fg\|_{\dot{H}^{\alpha}}\leq C\|f\|_{\dot{H}^{\alpha_1}}\|g\|_{\dot{H}^{\alpha_2}},
\end{equation*}
where $\alpha=\alpha_1+\alpha_2-\frac n2$.
\end{proposition}

\section{Main Results}
\begin{theorem}
Let $s\in [\frac12,1]$, $T>0$, $\alpha>\frac n2+1$, $u_{0}\in H^\alpha(\mathbb{R}^n)$, $v\in C([0,T];H^\alpha(\mathbb{R}^n))$, and $v\geq 0$. Then there is a unique solution $u\in C^1([0,T];H^\alpha(\mathbb{R}^n))$ to the linear initial value problem
\begin{equation}
\begin{cases}
\partial_t{u}=\nabla u\cdot \nabla(-\Delta)^{-s}v-v(-\Delta)^{1-s}u,\\
u(x,0)=u_{0}.
\end{cases}
\end{equation}
And if the initial data $u_0\geq 0$, we can get $u\geq0, (x,t)\in \mathbb{R}^n\times[0,T]$.
\end{theorem}

\begin{proof}
For any $\epsilon>0$, we consider the following linear problem
\begin{equation}
\begin{cases}
\partial_t{u^\epsilon}=F_\epsilon(u^\epsilon)=J_\epsilon(\nabla J_\epsilon u^\epsilon\cdot \nabla(-\Delta)^{-s} v)-J_\epsilon (v(-\Delta)^{1-s}J_\epsilon u^\epsilon),\\
u^\epsilon(x,0)=u_{0}.\end{cases}
\end{equation}
By Proposition 2.1, Proposition 2.2 and $s\geq \frac 12$ we can estimate
\begin{align*}
\|F_\epsilon(u_1^\epsilon)-F_\epsilon(u_2^\epsilon)\|_{H^\alpha}
&=\|J_\epsilon(\nabla J_\epsilon (u_1^\epsilon-u_2^\epsilon)\cdot \nabla(-\Delta)^{-s} v)-J_\epsilon (v(-\Delta)^{1-s}J_\epsilon (u_1^\epsilon-u_2^\epsilon)\|_{H^\alpha}\\
&\leq C(\epsilon,\|v\|_{H^\alpha})\|u_1^\epsilon-u_2^\epsilon\|_{H^\alpha}.
\end{align*}
By Picard Theorem, for any $\alpha>\frac n2+1, \epsilon>0$, there exists a $T_\epsilon=T_\epsilon(u_*)>0$, problem (3.2) has a unique solution $u^\epsilon\in C^1([0,T_\epsilon);H^\alpha)$. By Proposition 2.1 and Proposition 2.3,
\begin{align*}
\frac 12 \frac {d}{dt}\|u^\epsilon\|^2_{L^2}
&=\int \nabla J_\epsilon u^\epsilon \cdot\nabla(-\Delta)^{-s}vJ_\epsilon u^\epsilon-
\int v(-\Delta)^{1-s}J_\epsilon u^\epsilon J_\epsilon u^\epsilon\\
&\leq \frac 12\int\nabla |J_\epsilon u^\epsilon|^2\cdot\nabla(-\Delta)^{-s}v-\frac 12\int v(-\Delta)^{1-s}|J_\epsilon u^\epsilon|^2\\
&\leq  \frac 12\int|J_\epsilon u^\epsilon|^2(-\Delta)^{1-s}v-\frac 12\int|J_\epsilon u^\epsilon|^2(-\Delta)^{1-s}v=0.
\end{align*}
Moreover, for any $\alpha>0$,
\begin{align*}
\frac 12 \frac {d}{dt}\|\Lambda^\alpha u^\epsilon\|^2_{L^2}
&=\int \Lambda^\alpha(\nabla J_\epsilon u^\epsilon \cdot\nabla(-\Delta)^{-s}v)J_\epsilon\Lambda^\alpha u^\epsilon-
\int \Lambda^\alpha(v(-\Delta)^{1-s}J_\epsilon u^\epsilon) \Lambda^\alpha J_\epsilon u^\epsilon\\
&\leq C\|[\Lambda^\alpha,\nabla(-\Delta)^{-s}v]\nabla J_\epsilon u^\epsilon\|_{L^2}\|\Lambda^\alpha u^\epsilon\|_{L^2}
+\int\nabla(-\Delta)^{-s}v\Lambda^\alpha \nabla J_\epsilon u^\epsilon\Lambda^\alpha J_\epsilon u^\epsilon\\
&+ C\|[\Lambda^\alpha,v](-\Delta)^{1-s}J_\epsilon u^\epsilon\|_{L^2}\|\Lambda^\alpha u^\epsilon\|_{L^2}-\int v\Lambda^\alpha (-\Delta)^{1-s}J_\epsilon u^\epsilon \Lambda^\alpha J_\epsilon u^\epsilon
\end{align*}
By Proposition 2.2 and Sobolev embedding,
\begin{align*}
\|[\Lambda^\alpha,\nabla(-\Delta)^{-s}v]\nabla J_\epsilon u^\epsilon\|_{L^2}&\leq C\|(-\Delta)^{1-s}v\|_{L^\infty}\|\Lambda^{\alpha-1} \nabla J_\epsilon u^\epsilon\|_{L^2}+
\|\nabla(-\Delta)^{-s}v\|_{\dot{H}^\alpha}\|\nabla J_\epsilon u^\epsilon\|_{L^\infty}\\
&\leq C\|v\|_{H^\alpha}\|u^\epsilon\|_{H^\alpha},
\end{align*}
\begin{align*}
\|[\Lambda^\alpha,v](-\Delta)^{1-s}J_\epsilon u^\epsilon\|_{L^2}&\leq C\|\nabla v\|_{L^\infty}\|(-\Delta)^{1-s}J_\epsilon u^\epsilon\|_{\dot{H}^{\alpha-1}}+
\|v\|_{\dot{H}^\alpha}\|(-\Delta)^{1-s}J_\epsilon u^\epsilon\|_{L^\infty}\\
&\leq C\|v\|_{H^\alpha}\|u^\epsilon\|_{H^\alpha}.
\end{align*}
By Proposition 2.3,
\begin{align*}
&\int\nabla(-\Delta)^{-s}v\Lambda^\alpha \nabla J_\epsilon u^\epsilon\Lambda^\alpha J_\epsilon u^\epsilon-\int v\Lambda^\alpha (-\Delta)^{1-s}J_\epsilon u^\epsilon \Lambda^\alpha J_\epsilon u^\epsilon\\
& \leq\frac12\int \nabla(-\Delta)^{-s}v\nabla(\Lambda^\alpha J_\epsilon u^\epsilon)^2-\frac12\int v(-\Delta)^{1-s}(\Lambda^\alpha J_\epsilon u^\epsilon)^2\\
&\leq C\|v\|_{H^\alpha}\|u^\epsilon\|^2_{H^\alpha}.
\end{align*}
Combine the above estimates,
\begin{equation*}
\frac d {dt}\|u^\epsilon(\cdot,t)\|_{H^{\alpha}}\leq C\|v\|_{H^\alpha}\|u^\epsilon\|_{H^\alpha}.
\end{equation*}
By Gronwall's inequality,
\begin{equation*}
\|u^\epsilon(\cdot,t)\|_{H^\alpha}\leq \|u_0\|_{H^\alpha}\exp(C\sup_{0\leq t\leq T}\|v\|_{H^\alpha}).
\end{equation*}
Such the solution $u^\epsilon$ exists on $[0,T]$.
Similarly,
\begin{equation*}
\frac d {dt}\|u^\epsilon(\cdot,t)\|_{H^{\alpha-1}}\leq C\|v\|_{H^\alpha}\|u^\epsilon\|_{H^\alpha}\leq C(\|v\|_{H^\alpha}, \|u_0\|_{H^\alpha}, T).
\end{equation*}
By Aubin compactness theorem, there is a subsequence of $\{u^{\frac 1n}\}_{n\geq1}$ that convergence strongly to $u$ in $C([0,T];H^\alpha)$. If $\alpha>\frac d2+1$,
$H^\alpha \hookrightarrow C^1$, so $u$ is a solution of (3.1).

If $u,\tilde{u}$ are two solutions of problem (3.1), then $w=u-\tilde{u}$ satisfies
\begin{equation*}
\begin{cases}
\partial_t{w}=\nabla w\cdot \nabla(-\Delta)^{-s}v-v(-\Delta)^{1-s}w,\\
w(x,0)=0.
\end{cases}
\end{equation*}
Similarly, we can get $\frac d{dt}\|w\|_{L^2}\leq 0$ and $\frac d{dt}\|w\|_{\dot{H}^\alpha}\leq \|v\|_{H^\alpha}\|w\|_{H^\alpha}$, i.e,
$\frac d{dt}\|w\|_{H^\alpha}\leq \|v\|_{H^\alpha}\|w\|_{H^\alpha}$. By Gronwall's inequality we can deduce $u(x,t)=0$, $(x,t)\in \mathbb{R}^n\times[0,T]$.

Since $u_0\geq 0$ then if there exists a first time $t_0$ where for some point $x_0$ we have $u(x_0,t_0)=0$, then $(x_0,t_0)$ will correspond to a minimum point
and therefore $\nabla u(x_0,t_0)=0$, and
\begin{equation*}
(-\Delta)^{1-s}u(x)=c\int \frac {u(x)-u(y)}{|y|^{n+2-2s}}dy\leq0.
\end{equation*}
 Hence $ u_t|_{(x_0,t_0)}\geq 0$. So $u(x,t)\geq 0$ for all $(x,t)\in \mathbb{R}^n\times[0,T]$.
\end{proof}

\begin{theorem}
Let $n\geq 2$, $s\in [\frac12,1)$, $\alpha>\frac d2+1$, $u_0\in H^\alpha(\mathbb{R}^n)$, and $u_0\geq 0$. Then there is a unique solution $u\in C^1([0,T_0],H^\alpha(\mathbb{R}^n))$ to the linear initial value problem
\begin{equation*}
\begin{cases}
\partial_t{u}=\nabla\cdot(u\nabla(-\Delta)^{-s}u),\\
u(x,0)=u_{0}.
\end{cases}
\end{equation*}
And if the initial data $u_0\geq 0$, we can get $u\geq0, (x,t)\in \mathbb{R}^n\times[0,T_0]$.
\end{theorem}
\begin{proof}
Set $u^1=u_0$. Notice $\partial_t{u}=\nabla\cdot(u\nabla(-\Delta)^{-s}u)=\nabla u\cdot \nabla (-\Delta)^{-s}u-u(-\Delta)^{1-s}u$,
 we construct a sequence $\{u^{n}\}$ defined by solving the following systems
\begin{equation}\label{3.1}
\begin{cases}
\partial_t{u^{n+1}}=\nabla u^{n+1}\cdot \nabla(-\Delta)^{-s}u^{n}-u^n(-\Delta)^{1-s}u^{(n+1)},\\
u^{n+1}(x,0)=u_{0}.
\end{cases}
\end{equation}
Firstly by Theorem 3.1, we get $u^2\in C([0,T);H^\alpha)$, $\forall T<\infty$, and it satisfies $u^2\geq 0$ and
\begin{equation*}
\sup_{0\leq t\leq T}\|u^2\|_{H^\alpha}\leq \|u_0\|_{H^\alpha}\exp(C\|u^1\|_{H^\alpha}T).
\end{equation*}
If $\exp(2C\|u_1\|_{H^\alpha}T_0)\leq2$, for example $T_0=\frac {\ln 2}{2C(1+\|u_0\|_{H^\alpha})}$, we have $\sup_{0\leq t\leq T_0}\|u^2\|_{H^\alpha}\leq 2\|u_0\|_{H^\alpha}$. By the standard induction argument, if $u^{n}\in C([0,T_0];H^\alpha)$ , $u^n\geq 0$ is a solution of (\ref{3.1}) with $\|u^n\|_{H^\alpha}\leq 2\|u_0\|_{H^\alpha}$, by Theorem 3.1 we can get $u^{n+1}\in C([0,T_0];H^\alpha)$ , $u^{n+1}\geq0$ and
\begin{equation*}
\sup_{0\leq t\leq T_0}\|u^{n+1}\|_{H^\alpha}\leq \|u_0\|_{H^\alpha}\exp(C\|u^{n}\|_{H^\alpha}T_0)\leq 2\|u_0\|_{H^\alpha},
\end{equation*}
\begin{equation*}
\frac d{dt}\|u^{n+1}\|_{H^{\alpha-1}}\leq C\|u^n\|_{H^\alpha}\|u^{n+1}\|_{H^\alpha}\leq C\|u_0\|^2_{H^\alpha}.
\end{equation*}
By Aubin compactness theorem, there is a subsequence of $u^{n}$ that convergence strongly to $u$ in $C([0,T];H^\alpha)$. If $u\geq 0,\tilde{u}\geq0$ are two solutions of problem (3.1), then $w=u-\tilde{u}$ satisfies that
\begin{equation*}
\begin{cases}
\partial_t{w}=\nabla\cdot(w \nabla(-\Delta)^{-s}u)+\nabla \cdot(\tilde{u}\nabla(-\Delta)^{-s}w),\\
w(x,0)=0.
\end{cases}
\end{equation*}
By Proposition 2.2 we can get
\begin{align*}
\frac12\frac d{dt}{\|w\|_{L^2}^2}&=\int w\nabla\cdot(w \nabla(-\Delta)^{-s}u)+\int w\nabla \tilde{u}\cdot\nabla(-\Delta)^{-s}w-\int w \tilde{u}(-\Delta)^{1-s}w\triangleq I_1+I_2+I_3.
\end{align*}
$I_1,I_3$ can be estimated as:
\begin{align*}
&I_1=\int w\nabla w\cdot \nabla(-\Delta)^{-s}u=\frac12\int\nabla w^2\cdot \nabla(-\Delta)^{-s}u=\frac12\int w^2(-\Delta)^{1-s}u\leq C\|u\|_{H^{\alpha}}\|w\|^2_{L^2},\\
&I_3\leq-\frac 12\int \tilde{u}(-\Delta)^{1-s}w^2=\int-\frac 12(-\Delta)^{1-s}\tilde{u}\cdot w^2\leq C\|u\|_{H^{\alpha}}\|w\|^2_{L^2}.
\end{align*}
When $s>\frac 12$,
\begin{align*}
I_2&\leq C\|w\|_{L^2}\|\nabla u\cdot\nabla(-\Delta)^{-s}w\|_{L^2}        \\
   &\leq C\|w\|_{L^2}\|\nabla u\|_{\dot{H}^{\frac n2+1-2s}}\|\nabla(-\Delta)^{-s}w\|_{\dot{H}^{2s-1}}
   \leq C\|u\|_{H^{\alpha}}\|w\|^2_{L^2}.
\end{align*}
When $s=\frac 12$, the above estimates are still valid.  Combine the above estimates, $\frac {d}{dt}\|w\|_{L^2}\leq C\|w\|_{L^2}\|u\|_{H^{\alpha}}$. By Gronwall's inequality we can deduce $w(x,t)=0$ on $[0,T_0]$.
\end{proof}

\section{Correction}
In our previous paper \cite{zhou01} in which we establish the well-posedness result in Besov spaces for the equation (\ref{equation01}), there is a mistake in page 9 when we estimate the term $J_4'$ in equation (4.5). To correct the mistake, we modify our proof in the way as following: First we construct the approximate equation: 
\begin{equation}\label{equation41}
\begin{cases}
 u_t^{(n+1)}=\nabla u^{(n+1)} \cdot \nabla(-\triangle)^{-s}u_{\epsilon}^{(n)}-u_{\epsilon}^{(n)}(-\triangle)^{1-s}u^{(n+1)};   \\
 u^{(n+1)}(0)=\sigma_{\epsilon}*u_0,\quad u^{(1)}=\sigma_{\epsilon}*u_0.
\end{cases}
\end{equation}
By the argument in section 2, we can always find the sequence $u^{(n)}$ who solves the linear systems (\ref{equation41}). Assume $u_0\geq0$, we prove $u^{(n+1)}\geq0$.
Inspired by \cite{caffarelli02}, we assume that $x_0$ is a point of minimum of $u^{(n+1)}$ at time $t=t_0$. This indicate that
$\nabla u^{(n+1)}(x_0)=0$, and 
\[
 (-\triangle)^{1-s} u^{(n+1)}(x_0)=c\int \frac{u(x_0)-u(y)}{|y|^{n+2(1-s)}}dy\leq0.
\]
Thus we deduce $\frac{\partial}{\partial t}u^{(n+1)}\big|_{t=t_0}\geq0$, and by induction there holds
     $$u^{(n+1)}\geq 0.$$
By the same way as \cite{zhou01}, taking $\triangle_j$ on (\ref{equation41}), we obtain
\begin{align*}
 \partial_t \triangle_j u^{(n+1)}
& =\sum[\triangle_j,\partial_i (-\triangle)^{-s} u_{\epsilon}^{(n)}] \partial_i u^{{n+1}}
   +\sum \partial_i(-\triangle)^{-s} u_{\epsilon}^{(n)} \triangle_j (\partial_i u^{(n+1)})	\\
&-[\triangle_j,u_{\epsilon}^{(n)}] (-\triangle)^{1-s} u^{(n+1)}
  -u_{\epsilon}^{(n)} \triangle_j (-\triangle)^{1-s}u^{(n+1)}.
\end{align*}
Multiplying both sides by $\frac{\triangle_j u^{(n+1)}}{|\triangle_j u^{(n+1)}|}$, and integrating over 
$\mathbb{R}^d$, then denote the corresponding part in the right side by $J_1', J_2', J_3', J_4'$, respectively. We obtain the estimates,
\begin{align*}
 J_1' \leq C 2^{-j\alpha}\|u^{(n+1)}\|_{B_{1,\infty}^{\alpha}} \|u^{(n)}\|_{B_{1,\infty}^{\alpha+1-2s}}	\\
 J_2' \leq C 2^{-j\alpha}\|u^{(n+1)}\|_{B_{1,\infty}^{\alpha}} \|u^{(n)}\|_{B_{1,\infty}^{\alpha+1-2s}}	\\
 J_3' \leq C 2^{-j\alpha}\|u^{(n)}\|_{B_{1,\infty}^{\alpha}} \|u^{(n+1)}\|_{B_{1,\infty}^{\alpha+1-2s}}.
\end{align*}
And the estimate for the term $J_4'$ is replaced by  
\begin{align*}
J_4' =
 &-\int u^{(n)}\triangle_j (-\triangle)^{1-s} u^{(n+1)} \frac{\triangle_j u^{(n+1)}}{|\triangle_j u^{(n+1)}|} \\
 &\leq -\int u^{(n)} (-\triangle)^{1-s} |\triangle_j u^{(n+1)}|      \\
 &\leq -\int (-\triangle)^{1-s} u^{(n)}  |\triangle u^{(n+1)}|   \\
 &\leq 2^{-j\alpha}\|u^{n}\|_{B_{1,\infty}^{r+2-2s}}\|u^{(n+1)}\|_{B_{1,\infty}^{\alpha}}.
\end{align*}
Here $r>d$ be a arbitrary real number, and in the first inequality we use the following pointwise estimate:
\begin{proposition}\cite{miao01}
 Set $0\leq \alpha \leq 2$, $p\geq1$. Then for any $f\in\mathcal{S}(\mathbb{R}^d)$, there holds
 $$p|f(x)|^{p-2}f(x)\Lambda^{\alpha}f(x)\geq \Lambda^{\alpha}|f(x)|^p.$$
\end{proposition}
\noindent
Taking $r$ such that $r+2-2s<\alpha$, e.g. set $r=\alpha-1$, we conclude 
\[
  \frac{d}{dt}\|u^{(n+1)}\|_{B_{1,,\infty}^\alpha}\leq \|u^{(n)}\|_{B_{1,,\infty}^\alpha}\|u^{(n+1)}\|_{B_{1,,\infty}^\alpha}.
\]
The other parts of the proof are no difference.

\section*{Acknowledgement}
We are very grateful for PhD. Mitia Duerinckx for pointing out our mistake and share some good views on this problem. This paper is supported by the NNSF of China under grants No. 11601223 and No.11626213.


\vskip 10mm
\noindent
$^1$ \  Department of Information Technology, Nanjing Forest Police College,  210023 Nanjing, China\\
\indent Email: zhouxuhuan@163.com \\
$^{2*}$ \ School of Applied Mathematics,  Nanjing University of Finance and Economics, 210023 Nanjing,
China\\
\indent Email: xwltc123@163.com\\
\end{CJK*}
\end{document}